%% file: main.tex
\documentclass[sigconf,balance=false]{acmart}
\usepackage{popets}
\usepackage{float}
\usepackage{graphicx}
\usepackage{multirow}
\usepackage{xspace}
\usepackage{amsmath}
\usepackage{arabtex}
\usepackage{tikz}
\usepackage{utf8}
\usepackage[T1]{fontenc}
\usepackage[utf8]{inputenc}
\usepackage{xspace} 
\usepackage{balance}
\usepackage{enumitem}
\usepackage{xcolor}
\usepackage{graphicx}
\usepackage{xparse}
\usepackage{amsmath}
\usepackage{anyfontsize}
\usepackage{color}
\usepackage{rotating}
\usepackage{adjustbox}
\usepackage{colortbl}
\usepackage{microtype}
\usepackage{etoolbox}
\usepackage{multicol}
\usepackage[english]{babel}
\usepackage{pifont}
\usepackage[flushleft]{threeparttable}
\usepackage{array,booktabs,makecell}
\usepackage[font=small]{caption}
\usepackage{tabularx}

\newtheorem{remark}{Remark}[section] 

\makeatletter
\patchcmd{\@makecaption}
  {\scshape}
  {}
  {}
  {}
\makeatother
	
\definecolor{lightcyan}{rgb}{0.88,1,1}
\definecolor{cadetgrey}{rgb}{0.57, 0.64, 0.69}
\definecolor{airforceblue}{rgb}{0.36, 0.64, 0.66}

\newrobustcmd*{\mycircle}[1]{\tikz{\filldraw[draw=#1,fill=#1] (0,0) circle [radius=0.1cm];}}

\newrobustcmd*{\mytriangle}[1]{\tikz{\filldraw[draw=#1,fill=#1] (0,0) --
(0.2cm,0) -- (0.1cm,0.2cm);}}

\colorlet{pastelgreen}{green!}
\colorlet{pastelred}{red!50!}
\colorlet{pastelyellow}{yellow!50!}
\definecolor{calpolypomonagreen}{rgb}{0.12, 0.3, 0.17}

\usepackage[strings]{underscore}
\usepackage{array}
\usepackage{soul}

\newcommand{\myparagraph}[1]{\smallskip\textbf{\textit{#1}:}\hspace{3pt}}

\setcopyright{none} 
\copyrightyear{}

\acmYear{} 
\acmVolume{} 
\acmNumber{} 
\acmDOI{} 
\acmISBN{}
\acmConference{Research Article} 
\begin{document}

\title[Integer Factorization: Another perspective]{Integer Factorization: Another perspective}


\author{
    {\rm Gilda Rech BANSIMBA$^{*}$ \hspace{1em} Regis Freguin BABINDAMANA$^{\dagger}$\\ 
     $^{*}$bansimbagilda@gmail.com \hspace{1em} $^\dagger$regis.babindamana@umng.cg
}}

\renewcommand{\shortauthors}{Gilda. R. Bansimba. Regis. F. Babindamana}
\begin{abstract}
  Integer factorization is a fundamental problem in algorithmic number theory and computer science. 
This problem is considered as a one way or trapdoor function in cryptography, mainly in the Rivest, Shamir and Adleman (RSA) cryptosystem where it constitutes the foundation of the key exchange protocol and digital signature.
To date, from elementary trial division to sophisticated methods like the General Number Field Sieve, no known algorithm can break the problem in polynomial time, while its proved that Shor's algorithm could on a quantum computer. That said, these methods are limited in term of efficiency and their success strongly linked to computational resources, particularly with the shor's algorithm using quantum circuits that entirely depends on a quantum computer.\\
In this paper, we start by recalling slightly some factorization facts and then approach the problem under completely different angles. Firstly, we take the problem from the ring $\displaystyle\left(\mathbb{Z}, \text{+}, \cdot\right)$ to the Lebesgue space $\mathcal{L}^{1}\left(X\right)$ where $X$ can be $\mathbb{Q}$ or any given interval setting. From this first perspective, integer factorization becomes equivalent to finding the perimeter of a rectangle whose area is known. In this case, it is equivalent to either finding bounds of integrals or finding primitives for some given bounds. 
Secondly, we take the problem from the ring $\displaystyle\left(\mathbb{Z}, \text{+}, \cdot\right) $ to the ring of matrices $\left( M_{2}\text{(}\mathbb{Z}\text{)}, \ \text{+} \ \cdot\right)$ and show that this problem is equivalent to matrix decomposition, and therefore present some possible computing algorithms, particularly using Gröbner basis and through matrix diagonalization.
Finally, we address the problem depending on algebraic forms of factors and show that this problem is equivalent to finding small roots of a bivariate polynomial through coppersmith's method.

The aim of this study is to propose innovative methodological approaches to reformulate this problem, thereby offering new perspectives. 

\end{abstract}


\maketitle

\input{introduction}

\input{background}
\input{contributions}


\input{conclusion}
In this paper, we have recalled some factorization methods and then approached the problem from completely different angles. First, we have considered the problem from the ring $\displaystyle\left(\mathbb{Z}, \text{+}, \cdot\right)$ to Lebesgue space $\mathcal{L}^{1}\left(X\right)$ where $X$ can be $\mathbb{Q}$ or any given interval setting, in which case, this problem is equivalent to finding the perimeter of a rectangle whose area is known. In this direction, it is equivalent to either finding bounds of integrals or finding primitives for some given bounds. 
Secondly, we have taken the problem from the ring $\displaystyle\left(\mathbb{Z}, \text{+}, \cdot\right) $ to $\left( M_{n}\text{(}\mathbb{Z}\text{)}, \ \text{+} \ \cdot\right)$ and show that this problem is equivalent to matrix decomposition, and therefore present a computing algorithm through a Gröbner basis computation approach and matrix diagonalization.
Finally, we have addressed the problem depending on algebraic forms of factors and show that this problem is equivalent to finding small roots of a bivariate polynomial through coppersmith's method.
The aim of this article is to present the problem from different angles, opening up new directions for approaching the problem in a different way taking advantage of measure, calculus and related theories.\\
As future perspective, for the MAFpv, a future immediate work would be investigating geometrically and finding a possible construction rule on these hyperbolas to recover the sought solutions with respect to the equation coefficients and divisors of $n-1$. In addition, best upper bounds $X$ and $Y$, and $M$ should be investigated for better performances. 
\section{Acknowledgements}
The authors would like to thank the anonymous reviewers for their helpful comments on the paper.. \\ \\

\bibliographystyle{plain}

\end{document}

%% file: introduction.tex
\section{Introduction}
Factoring integers is one of the most important problems in cryptology and number theory. In 1801, in "Disquisitiones Arithmeticae" \cite{gauss}, Gauss identified integer factorization and the primality testing as the two fundamental problems of arithmetic. Indeed, he went on to say: "The dignity of science itself seems to require that all possible means should be explored to solve a so
elegant and famous problem." (Translation of the original texts). 
These two main problems are at the heart of cryptographic protocols, to ensure the confidentiality, authenticity, integrity and non-repudiation of information. This is the case of the Diffie-Hellman Elliptic Curve (ECDH) protocol implemented in TLS, for example, where the primality of the cardinal of the associated group is an important security feature to avoid attacks like MOV; the Elliptic Curve Digital Signature (ECDSA) protocol, or the RSA cryptosystem, whose security essentially relies on integer factorization. 


Over the 20th century, a number of factoring algorithms were proposed, the most important of which are the Quadratic Sieve (QS) proposed by Carl Pomerance \cite{qs} in 1985 and the General Number Field Sieve (GNFS) proposed by the same author 9 years later \cite{nfs} in 1994, based on the Factorisation method proposed by Pierre de Fermat a century earlier \cite{fer1}, \cite{fer2}, in addition the Elliptic Curve Method (ECM) proposed by Hendrik Lenstra \cite{ecm} in 1987. These methods have also seen significant improvements during this century (see \cite{dv}, \cite{mont}, \cite{dt2}, \cite{dt8}, \cite{dt6}, \cite{improv1}, \cite{improv2}).
Since the beginning of the 21st century, several other improvements to these methods have been proposed, particularly for the ECM method, both in terms of parametrizations and addition formulas for different coordinate systems ( \cite{ed}, \cite{eced}, \cite{dt1}, \cite{dt3}, \cite{dt4}, \cite{dt5}, \cite{dt6}, \cite{gmp}, \cite{mpfq}). 

In practice however, these algorithms still present feasibility limits for fairly large parameters. This is the case, for example, of the record set in 2019 by Emmanuel Thomé, Paul Zimmermann, Fabrice Boudot, Pierrick Gaudry, Aurore Guillevic and Nadia Heninger on a 795-bit RSA module (RSA-240) using an optimized implementation of the state-of-the-art method, the number-field sieve method in the CADO-FS software \cite{kdo} with a computation time of around 1000 physical core-years on Intel Xeon Gold 6130 CPUs \cite{kdo1}, or even the latest feat on an 829-bit RSA module (RSA-250) achieved in February 2020 by the same authors using CADO-NFS with a computation time of 2,700 core-years on 
Intel Xeon Gold 6130 CPUs, distributed across the computing resources of the INRIA and CNRS laboratories and their partner institutions; the EXPLOR computing center of the 
Université de Laurraine in Nancy (France); the PRACE research infrastructure using resources at the Juelich supercomputing center in Germany, as well as cisco equipment at UCSD. 

All these requirements in terms of hardware resources and very high computing power demonstrate how important it is to explore new directions in order to propose new approaches that could lead to efficiently solving such an important problem in cryptology and number theory.

In this paper we address the problem in a another direction with a whole new point of view. That is, primitives and integral calculus on one side,  matrix decomposition with application of gröbner basis on the other side and finally polynomial small roots computation through coopersmith method.

%% file: background.tex
\section{Background}
\label{sec:background}

Factorization methods can be divided into 2 main categories: special purpose factorization methods and general purpose factorization methods.\\
We hereby consider $n$ to be a semiprime of size $b$ bits, $p$ and $q$ factors of $n$ with $\min \left\lbrace  p, q \right\rbrace =p$.
\subsection{Special purpose factorization methods}
Special purpose factorization methods running time depends on the properties of the number to be factored, mainly its factors such as the smallest prime factor. Among these algorithms, we have:
\begin{enumerate}
\item Pollard's rho algorithm \\
The Pollard's rho algorithm is an algorithm for integer factorization invented by John Pollard in 1975. 
It takes as its input $n$ the integer to be factored and considers the polynomial $g\left(x\right)= \left(x^2 \text{-}1 \right) \mod n $ or more commonly $g\left(x\right)= \left(x^2 \text{+}1 \right) \mod n $ and outputs a non trivial factor of $n$, or it fails.\\
This method has an expected running time proportional to the square root of the smallest prime factor of the composite number being factorized, that's to say $\displaystyle O\left(\sqrt{p}\right)$.\\

\item Fermat's factorization Method \ \\
The Fermat's Factorization Method factors an odd integer 
$n$ by expressing it as the difference of two squares:\\ $n=a^{2}\text{-}b^{2}=\left(a\text{-}b\right)\left(a\text{+}b\right)$.\\ Its time complexity in the worse case is $\displaystyle O\left(\sqrt[4]{n} \right)$. \\
\item Euler's factorization Method \\
The Euler's Factorization Method is based on representing a number $n$ as the sum of two squares in two different ways, leading to its factorization. Specifically, it uses the identity $n=a^2\text{+}b^2=c^2\text{+}d^2$.\\
Its time complexity in the worst case is $O\left(\sqrt{p} \right)$. \\
\item Lenstra's elliptic curve method \ \\
The Elliptic Curve factorization Method (ECM) that was introduced by H. Lenstra \cite{ecm} and conjectured to be in\\ $\displaystyle \mathsf{O}\left(\exp^{(\sqrt{2}\text{+}\mathsf{O}\text{(}1\text{)})\sqrt{ \log p \log\log p}}\right)$ (where $p$ is the smallest factor of the composite module $n$)
is an improvement of the pollard $p\text{-}1$ method (see \cite{pol}) where the multiplicative group $\mathbb{Z}\text{/}n\mathbb{Z}$ is replaced by $E\text{(}\mathbb{Z}\text{/}n\mathbb{Z}\text{)}$, an elliptic curve defined over the residue ring $\mathbb{Z}\text{/}n\mathbb{Z}$, where $n$ is the number to factor.\\ 
From short Weierstrass to the use of Montgomery curves \cite{mont}. Since the introduction of Edwards Curve \cite{eced} and their use  in the ECM \cite{ed} thanks to their large torsion subgroups by Mazur Theorem \cite{maz}, Edwards Curves have replaced Montgomery curves in lattest implementation ECM-MPFQ \cite{mpfq} over the GMP-ECM \cite{gmp}. Results in \cite{dt1, dt2, dt3, dt4, dt5, dt6, dt7, dt8}, show significant improvements not only at the scalar multiplication level with fewer number of Inversions, Multiplications and Squarings but also at parametrization level.\\
An elliptic curve over a field $\mathbb{K}$ is an algebraic curve given by a non singular equation of the form\\
 $y^{2}\text{+}a_{1}xy\text{+}a_{3}y=x^{3}\text{+}a_{2}x^{2}\text{+}a_{4}x\text{+}a_{6}, a_{i}\in \mathbb{K}$ called generalized weierstrass equation form. In the projective space $\mathbb{P}^{2}\text{(}\mathbb{K}\text{)}=\displaystyle \frac{\mathbb{K}^{3}\setminus \lbrace (0, 0, 0)\rbrace}{\sim}$ through the map $ \left(x, y\right)\longmapsto \left(X\text{/}Z, Y\text{/}Z\right)$, the equation becomes\\ $Y^{2}Z\text{+}a_{1}XYZ\text{+}a_{3}YZ^{2}=X^{3}\text{+}a_{2}X^{2}Z\text{+}a_{4}XZ^{2}\text{+}a_{6}Z^{3}$.\\
More generally, there exist isomorphisms that change the generalized form to reduced forms dependently on the characteristic of the field in which the curve is defined. For example 
\begin{eqnarray}
\left\lbrace
\begin{array}{ll}
y^{2}=x^{3}\text{+}ax\text{+}b \ if \ char(\mathbb{K})>3 \\
y^{2}=x^{3}\text{+}ax^{2}\text{+}bx\text{+}c \ if \ char(\mathbb{K})=3 \\
y^{2}\text{+}cy=x^{3}\text{+}ax\text{+}b \ if \ char(\mathbb{K})=2 \\
y^{2}\text{+}xy=x^{3}\text{+}ax^{2}\text{+}b \ if \ char(\mathbb{K})=2 \\
\end{array} \right.
\end{eqnarray}
$E\text{(}\mathbb{F}_{p}\text{)}=\lbrace (x, y)\in \mathbb{F}_{p}^{2} \ \text{:} \ y^{2}\text{+}a_{1}xy\text{+}a_{3}y=x^{3}\text{+}a_{2}x^{2}\text{+}a_{4}x\text{+}a_{6},a_{i}\in \mathbb{F}_{p} \rbrace \cup \lbrace \mathcal{O} \rbrace$, has an abelian group structure with the additive law where $\mathcal{O}=(0:1:0)$ denotes the point at infinity, corresponding to the case $Z=0$ in the projective space. \\ 
Considering a field with characteristic $p>3$ and two points $P=\text{(}x_{p}, y_{p}\text{)}, Q=\text{(}x_{q}, y_{q}\text{)} \in E_{a, b}\text{(}\mathbb{F}_{p}\text{)}=\lbrace (x, y)\in \mathbb{F}_{p}^{2} \ / \ y^{2}=x^{3}\text{+}ax\text{+}b, \ a, b\in \mathbb{F}_{p} \rbrace$, using the chord and tangent rule with the theorem of Bezout \cite{bezou}, the addition law on the curve is defined by
\begin{eqnarray*}
if \ P \neq Q:
\left\lbrace
\begin{array}{ll}
x_{P\text{+}Q}=\left( \frac{y_{P}-y_{Q}}{x_{P}-x_{Q}}\right) ^{2}- x_{P} -x_{Q}\\
y_{P\text{+}Q}=-y_{P}+\frac{y_{P}-y_{Q}}{x_{P}-x_{Q}}\left(x_{P} -x_{P+Q}\right)\\
\end{array} \right. ;
\end{eqnarray*}

\begin{eqnarray*}
 if \ P=Q:
\left\lbrace
\begin{array}{ll}
x_{P\text{+}P}=x_{2P}=\left(\frac{3x_{P}^{2}\text{+}a}{2y_{P}}\right)^{2}-2 x_{P}\\
y_{P\text{+}P}=y_{2P}=-y_{P}\text{+}\frac{3x_{P}^{2}\text{+}a}{2y_{P}}\left(x_{P} -x_{2P}\right)\\
\end{array} \right.
\end{eqnarray*}
Setting $\alpha =\frac{y_{P}-y_{Q}}{x_{P}-x_{Q}}$ or $\alpha = \frac{dy}{dx}\left(P\right)= \frac{3x_{P}^{2}\text{+}a}{2y_{P}}$ respectively while adding or doubling on the curve, we clearly see that these computations involve not only a certain number of Multiplications (M) and Squarings (S), but also Inversions (I).
Since computations are done in $\mathbb{Z}\text{/}n\mathbb{Z}$ with $n=p\times q$, $p, q$ primes, denominators $x_{P}-x_{Q}$ and $2y_{P}$ of $\alpha$ are invertible if and only if $\gcd(x_{P}-x_{Q}, n)=1$ when adding, or $\gcd(2y_{P}, n)=1$ when doubling. If not equal to $1$ and $\neq n$ then a proper divisor or a non-trivial factor of $n$ is found. This is the all idea of ECM, using the isomorphism between $E\left(\mathbb{Z}\text{/}n\mathbb{Z}\right)$ and $E\left(\mathbb{F}_{p}\right) \times E\left( \mathbb{F}_{q}\right)$. \\
Given a bound $B$, a point $P\in E\left(\mathbb{Z}\text{/}n\mathbb{Z}\right)$ and a $B-$smooth scalar $k$, the method consists on computing $kP$ until finding $kP=\mathcal{O}$ using double and add method or NAF (see \cite{speed1, speed2}) to speed-up computations, (see \cite{dt1, dt3, dt4, dt5} for improvements, and parametrizations yielding fewer number of M, S, and I)\\
For performances purpose, large torsion subgroup elliptic curves over rationals present much more advantages. Some families and parametrizations of curves present better performance compared to the conventional short weiestrass equation in characteristic greater than $3$. By Mazur theorem, every elliptic curve has a group torsion subgroup isomorphic to either one of $E_{tors}\text{(}\mathbb{Q}\text{)} \cong \mathbb{Z}\text{/}n_{i}\mathbb{Z}$,  $ n_{i}=1, 2, 3, 4, 5, 6, 7, 8, 9, 10$ or $12$ or $E_{tors}\text{(}\mathbb{Q}\text{)} \cong \mathbb{Z}\text{/}2\mathbb{Z}\times \mathbb{Z}\text{/}2n_{j}\mathbb{Z}$, $n_{j}=1, 2, 3$ or $4$.
So curves whose points have large torsion subgroup over $\mathbb{Q}$ are preferable. \\
For example over a field $\mathbb{K}$, the Montgomery curve given by $M_{A, B}: \  By^{2}=x^{3}\text{+}Ax^{2}\text{+}x$ where $B\text{(}A^{2}\text{-}4\text{)}\neq 0$ used in GMP-ECM implementation (see \cite{gmp}) with the law $\forall \ P=\text{(}x_{p}, y_{p}\text{)}, \ Q=\text{(}x_{q}, y_{q}\text{)} \in \lbrace \text{(}x, y\text{)}\in \mathbb{K}^2 \ / \ By^{2}=x^{3}\text{+}Ax^{2}\text{+}x \rbrace$,\\ $P\text{+}Q=\left(\frac{B\left(x_{q}y_{p}\text{-}x_{p}y_{q}\right)^2}{x_{p}x_{q}\text{(}x_{q}\text{-}x_{p}\text{)}^2}, \ \frac{\text{(}2x_{p}\text{+}x_{q}\text{+}A\text{)}\text{(}y_{q}\text{-}y_{p}\text{)}}{x_{q}\text{-}x_{p}}\text{-}\frac{B\text{(}y_{q}\text{-}y_{p}\text{)}^3}{(x_{q}\text{-}x_{p})^3}\text{-}y_{p} \right)$ and equivalent to the Weierstrass reduced form through the map $\varphi: \ M_{A, B} \longrightarrow \ E_{a, b}$, $\text{(}x, y\text{)}\longmapsto \text{(}x^{'}, y^{'}\text{)}=\text{(}\frac{x}{B}\text{+}\frac{A}{3B}, \frac{y}{B}\text{)}$ where $a=\frac{3\text{-}A^2}{3B^2}$, $b=\frac{2A^3\text{-}9A}{27B^3}$, and more recently Edwards Curves given by $x^{2}\text{+}y^{2}=1\text{+}dx^{2}y^{2}$ with $ d\setminus \lbrace 0, 1 \rbrace$ and generally twisted edwards curve given by $E_{a, d}: \ ax^{2}\text{+}y^{2}=1\text{+}dx^{2}y^{2}$, with $a^{5}\neq a, d\setminus \lbrace 0, 1 \rbrace$ used in the lattest implementation ECM-MPFQ (see \cite{mpfq}) present better performances for ECM. The Edward's addition law is defined by $\forall \ P=\text{(}x_{p}, y_{p}\text{)}, \ Q=\text{(}x_{q}, y_{q}\text{)} \in \lbrace \text{(}x, y\text{)}\in \mathbb{K}^2 \ \text{:} \ ax^{2}\text{+}y^{2}=1\text{+}dx^{2}y^{2} \rbrace$,\\ $P\text{+}Q=\left(\frac{1}{1\text{+}dx_{p}x_{q}y_{p}y_{q}}(x_{p}y_{q}\text{+}x_{q}y_{p}), \ \frac{1}{1\text{-}dx_{p}x_{q}y_{p}y_{q}}\text{(}y_{p}y_{q}\text{-}ax_{p}x_{q}\text{)} \right)$\\
which is equivalent to\\
$P\text{+}Q= \left( \frac{1}{ax_{p}x_{q}\text{+}y_{p}y_{q}}\text{(}x_{p}y_{p}\text{+}x_{q}y_{q}\text{)}, \ \frac{1}{x_{p}y_{q}\text{-}y_{p}x_{q}}\text{(}x_{p}y_{p}\text{-}x_{q}y_{q}\text{)}\right)$ (see \cite{edt}) where point doubling is not well defined and any Edward's equation form is birationally equivalent to the Montgomery form through the isomorphism $\Omega: \ E_{a,d}\longrightarrow \ M_{A,B}$,      $\text{(}x, y\text{)} \longmapsto \ \text{(}x^{'}, y^{'}\text{)}=\text{(} \frac{1\text{+}y}{1\text{-}y}, \ \frac{1\text{+}y}{x\text{(}1\text{-}y\text{)}}\text{)}$ where $A=\frac{2\text{(}a\text{+}d\text{)}}{a\text{-}d}$, $B=\frac{4}{a\text{-}d}$ and $ \Omega ^{-1}: \ M_{A, B} \longrightarrow \ E_{a, d}$, $\text{(}x^{'}, y^{'}\text{)} \longmapsto \ \text{(}\frac{x^{'}}{y^{'}},  \frac{x^{'}\text{-}1}{x^{'}\text{+}1}{\text{)}}$ where $a=\frac{A\text{+}2}{B}, \ d= \frac{A\text{-}2}{B}$.\\

\item Special number field sieve \\
The Special Number Field Sieve (SNFS) is one of the fastest known algorithms for factoring large integers of a special form, such as numbers of the form $a^n \pm b$. It uses algebraic number fields and a lattice reduction algorithm to find factors efficiently.\\
Its time complexity is expected to be\\ $O\left(e^{\left(1\text{+}O\text{(}1\text{)}\right)\left(\frac{32}{9}\log n\right)^{1\text{/}3}\left(\log\log n\right)^{2\text{/}3}}\right)$. \ \\

\end{enumerate}
\subsection{General purpose factorization methods}
This category of factorization methods has a running time entirely depending on the size of the integer to be factored.
Among, we have:
\begin{enumerate}
\item Trial division \\
The trial division factorization method consists on dividing the number $n$ by successive integers starting from the smallest prime ($2$) until the square root of $n$. If it is divisible by one of these integers, that integer is a factor of $n$.\\
Its worst case time complexity is $\displaystyle O\left(\sqrt{n}\right)$.\\
\item Dixon's factorization method \ \\
The Dixon's Factorization Method is a probabilistic algorithm that finds factors of a composite number $n$ by searching for numbers $x$ and $y$ such that $x^2 \equiv y^2 \mod n$. It uses a set of smooth numbers to find these congruences.\\
In the worst case, it has a time complexity of \\ $O\left(e^{2\sqrt{2}\sqrt{\log n \log\log n}}\right)=L_{n}\left[\frac{1}{2}, \ 2\sqrt{2} \right]$. \ \\

\item Continued fraction factorization method \ \\
The continued fraction factorization method (CFRAC) is a general purpose integer factorization algorithm, introduced by D. H. Lehmer and R. E. Powers in 1931 and developped by Michael A. Morrison and John Brillhart in 1975.\\
This method is based on Dixon's factorization method. It uses convergents in the regular continued fraction expansion of $\sqrt{kn}$ for any positive integer $k$.\\
It has a complexity of $\displaystyle\mathsf{O}\left(e^{\sqrt{2\log n \log \log n}} \right)=\displaystyle\mathsf{L}_{n}\left[1/2, \ \sqrt{2}\right]$. \\

\item Quadratic sieve \ \\
The Quadratic Sieve (QS) is a factorization algorithm that finds factors of a composite number $n$ by searching for numbers $x$ and $y$ such that $x^2 \equiv y^2 \mod n$. It uses a large set of smooth numbers to build a congruence of squares.\\
Its expected time complexity is $O\left(e^{\left(1\text{+}O\text{(}1\text{)}\right)\sqrt{\ln n \ln \ln n}} \right)$. \ \\

\item General number field sieve \ \\
The General Number Field Sieve (GNFS) is the most efficient classical algorithm for factoring large integers, particularly those with more than 100 digits. It uses algebraic number fields and polynomial selection to find relations that lead to a factorization of the target composite number $n$.\\
Its expected time complexity is\\ $O\left(e^{\left( \left(\left(\frac{8}{3}\right)^{2\text{/}3}\text{+} O\text{(}1 \text{)}\right)\left(\log n\right)^{1\text{/}3}\left(\log\log n\right)^{2\text{/}3} \right)} \right)=L_{n}\left[\frac{1}{3}, \ \left(64\text{/}9\right)^{1\text{/}3} \right]$.

\end{enumerate}


\myparagraph{Summary of some of the main factorization Methods} \ \\
In the same setting as above, we consider $n$ a semiprime of size $b$ bits, $p$ and $q$ its factors with $\min \left\lbrace  p, q \right\rbrace =p$. We then have the following summary table
\input{table_recap}

%% file: table_recap.tex
\begin{scriptsize}
\begin{table}[ht!]
  \caption{Integer factorization Methods by category}
  \centering 
  \begin{threeparttable}
    \begin{tabular}{ccccccc}
     \toprule
    Title  & \makecell{Special\\ purpose} & \makecell{General\\ Purpose} & Quantum & Complexity\\
    \cmidrule(l r ){1-5}
     Trial Division &   \ding{55} & \ding{51} & \ding{55} & $O\left(\displaystyle\sqrt{n}\right)$ \\
    \cmidrule(l  r ){1-5}
     Pollard's Rho & \ding{51} & \ding{55} & \ding{55} & $O\left(\sqrt{p}\right)$ \\ 
    \cmidrule(l r ){1-5}
     Fermat's Method & \ding{51} & \ding{55} & \ding{55} & $O\left(n^{1\text{/}4}\right)$ \\ 
    \cmidrule(l r ){1-5}
    Euler's Method   & \ding{51}   & \ding{55} & \ding{55} & $O\left(\sqrt{p}\right)$  \\
     \cmidrule(l r ){1-5}
    Lenstra's Method  & \ding{51}   & \ding{55} & \ding{55} & $O\left(e^{\left(\sqrt{2}\text{+}O\text{(}1\text{)}\right)\sqrt{ \log p \log\log p}}\right)$  \\
     \cmidrule(l r ){1-5}
     Special Number field sieve  & \ding{51}   & \ding{55} & \ding{55} & $O\left(e^{\left(1\text{+}O\text{(}1\text{)}\right)\left(\frac{32}{9}\log n\right)^{1\text{/}3}\left(\log\log n\right)^{2\text{/}3}}\right)$  \\
     \cmidrule(l r ){1-5}
    Shor's algorithm  & \ding{55}   & \ding{51}  & \ding{51} & $O\left(b^{3}\right)$  \\
     \cmidrule(l r ){1-5}
    Dixon's Method  & \ding{55}   & \ding{51}  & \ding{55} & $O\left(e^{2\sqrt{2}\sqrt{\log n \log\log n}}\right)$  \\
     \cmidrule(l r ){1-5}
    Quadratic Sieve  & \ding{55}   & \ding{51} & \ding{55} & $O\left(e^{\left(1\text{+}O\text{(}1\text{)}\right)\sqrt{\ln n \ln \ln n}} \right)$  \\
     \cmidrule(l r ){1-5}
    General Number Field Sieve  & \ding{55}   & \ding{51} & \ding{55} & $O\left(e^{\left( \left(\left(\frac{8}{3}\right)^{2\text{/}3}\text{+} O\text{(}1 \text{)}\right)\left(\log n\right)^{1\text{/}3}\left(\log\log n\right)^{2\text{/}3} \right)} \right)$  \\
    \bottomrule
    \end{tabular}
    \begin{tablenotes}
\item[*] Table of integer factorization methods by category.
\end{tablenotes}
\end{threeparttable}
  \end{table}
\end{scriptsize}

%% file: contributions.tex
\section{Contributions}
\label{sec:system}
The most difficult case in integer factorization is to factor a semi prime whose factors are big primes.\\ For this reason, in this section we consider the modulus $n$ to be a semiprime. 

\subsection{Rectangle point of view (Rpv)}

First and foremost, integer factorization can be reformulated as follows: 
the factorization of a semiprime $n=pq$ is equivalent to finding a rectangle's length and width whose area is known as $n$.\\
Equivalently, it is to find a perimeter of a rectangle whose area is known.\\
This is proved by the fact this case the perimeter $p=2(p+q)$ and area $n=pq$.

Then recovering $p$ and $q$ becomes equivalent to either finding bounds $p$ and $q$ such that $$\displaystyle\int_{0}^{p} \int_{0}^{q} \ d_{x}d_{y}=n$$, either finding primitives $\gamma_{1}(x)$ and $\gamma_{2}(y)$ such that for given bounds $\left(\alpha_{0},\alpha_{1}, \beta_{0}, \beta_{1}\right)\in \mathbb{Z}^{4}$, we have $$\displaystyle\int_{\alpha_{0}}^{\alpha_{1}} \int_{\beta_{0}}^{\beta_{1}}\gamma_{1}(x)\gamma_{2}(y) \ d_{x}d_{y}=n$$.
\begin{lemma} \label{lem1}\ \\ \rm{
For any prime number $p$ greater than 3, there exists a positive integer $x$ such that $p=y(x)$ is solution of the first order ordinary differential equation with Dirichlet condition $$ y^{'}=6, \ \text{with \ \ }  y(0)=\pm 1$$
}
\end{lemma}
\begin{proof} \ \\ \rm{
$$y^{'}=\frac{dy}{dx}=6 \ ,$$\\
then $\int dy=\int 6dx$. Thus $y=6x+c$ with $c=\pm 1$. Hence \\ the prime number $p=y(x)=6x \pm 1$. From a basic observation from elementary number theory using properties of integers and modular arithmetic, any prime number can be represented as such.
}
\end{proof}
\begin{example} \ \\ \rm{
Consider the following primes:
\begin{itemize}
\item[•] the 100 bits prime:\\
{\scriptsize $p=570303428823043591555786898897=y(x)=6x-1$},\\ with {\scriptsize $x=95050571470507265259297816483$}
\item[•] the $233$ bits prime:\\ 
{\scriptsize $p=6833702715893496540959299291382359232559264092523368475576074466071807=y(x)=6x-1$} with {\scriptsize $x=1138950452648916090159883215230393205426544015420561412596012411011968$}
\item[•] the $333$ bits prime:\\
{\scriptsize $p=6358337399459401417678277930704848782199223004051687909162273047126596686480950\\628594226066122559163=y(x)=6x+1$} with {\scriptsize $x=1059722899909900236279712988450808130\\366537167341947984860378841187766114413491771432371011020426527$}
\item[•] the $1000$ bits prime:\\
{\scriptsize $p=658089189840301730450496429519238792660987613697354796644846033244621236257652046\\021740484921759504122555940691871306258635428762907108947956854239712936130938554245\\542843969795828244216429256228793034041405027874070693186794739605844025450561774421\\858331525104830589953173100274244230949509794254643=y(x)=6x-1$} with {\scriptsize $x=109681531640050288408416071586539798776831268949559132774141005540770206042942007670\\290080820293250687092656781978551043105904793817851491326142373285489355156425707590\\473994965971374036071542704798839006900837979011782197799123267640670908426962403643\\055254184138431658862183379040705158251632375774$}
\end{itemize}
}
\end{example}
Let's consider the following proposition:
\begin{proposition} \ \\ \rm{
Given a RSA modulus $n\geqslant 35$, set $\gamma_{1}(x)=6$; $\gamma_{2}(y)=6$ and $\alpha_{0}=\beta_{0}=\pm \frac{1}{6}$ then there exists $\alpha$ and $\beta$ such that $$\displaystyle\int_{\pm \frac{1}{6}}^{\alpha} \int_{\pm \frac{1}{6}}^{\beta}\gamma_{1}(x)\gamma_{2}(y) \ d_{x}d_{y}=n$$.
}
\end{proposition}
\begin{proof} \ \\ \rm{
Using the Fubini-Tonelli \cite{fubini} theorem,
\begin{align*}
\displaystyle\int_{\pm \frac{1}{6}}^{\alpha} \int_{\pm \frac{1}{6}}^{\beta}\gamma_{1}(x)\gamma_{2}(y) \ d_{x}d_{y}&=\displaystyle\int_{\pm \frac{1}{6}}^{\alpha} \gamma_{1}(x)\ d_{x}\int_{\pm \frac{1}{6}}^{\beta}\gamma_{2}(y)d_{y}\\
&=\displaystyle\int_{\pm \frac{1}{6}}^{\alpha} (6)\ d_{x}\int_{\pm \frac{1}{6}}^{\beta}(6)d_{y}\\
&=\displaystyle\left[6x+ C_{1} \right]_{\pm \frac{1}{6}}^{\alpha}\displaystyle\left[6y+ C_{2} \right]_{\pm \frac{1}{6}}^{\beta}\\
&=\displaystyle\left(6\alpha\pm 1 \right) \left(6\beta\pm 1 \right) \\
&=y(\alpha)y(\beta) \text{, from lemma \ref{lem1}}\\
&=pq\\
&=n
\end{align*}
}
\end{proof}

\begin{example} \ \\ \rm{
Here we give two examples of 256 bits semiprimes. 
\begin{itemize}
\item
$\alpha=31184864858157931962611989461653189383$;\\ $\beta=56155612513944961522154793570341370125$
\begin{align*}
\displaystyle\int_{-\frac{1}{6}}^{\alpha} \int_{- \frac{1}{6}}^{\beta}\gamma_{1}(x)&\gamma_{2}(y) \ d_{x}d_{y}=\displaystyle\int_{- \frac{1}{6}}^{\alpha} \gamma_{1}(x)\ d_{x}\int_{- \frac{1}{6}}^{\beta}\gamma_{2}(y)d_{y}\\
&=\displaystyle\int_{- \frac{1}{6}}^{31184864858157931962611989461653189383} (6)\ d_{x}\\
&\times\int_{- \frac{1}{6}}^{56155612513944961522154793570341370125}(6)d_{y}\\
&=\displaystyle\left[6x+ C_{1} \right]_{- \frac{1}{6}}^{31184864858157931962611989461653189383}\\  
& \ \ \times\displaystyle\left[6y+ C_{2} \right]_{- \frac{1}{6}}^{56155612513944961522154793570341370125}\\
&=\displaystyle\left(6\cdot 31184864858157931962611989461653189383+ 1 \right)\\ 
&\times \left(6\cdot 56155612513944961522154793570341370125+ 1 \right) \\
&=y(31184864858157931962611989461653189383)\\
& \ \ \ \times  y(56155612513944961522154793570341370125)\\
& \text{, from lemma \ref{lem1}}\\
&=187109189148947591775671936769919136299\\
&\ \ \ \times 336933675083669769132928761422048220751\\
&=6304338674188041713811886561148660301051\\
& \ \ \ 9216638181159430642741072139609140549\\
&=n
\end{align*}
\item
$\alpha=42575376056348197982869887629770563187$;\\ $\beta=33263421842796264531163131543015665831$
\begin{align*}
\displaystyle\int_{\frac{1}{6}}^{\alpha} \int_{- \frac{1}{6}}^{\beta}\gamma_{1}(x)&\gamma_{2}(y) \ d_{x}d_{y}=\displaystyle\int_{\frac{1}{6}}^{\alpha} \gamma_{1}(x)\ d_{x}\int_{-\frac{1}{6}}^{\beta}\gamma_{2}(y)d_{y}\\
&=\displaystyle\int_{\frac{1}{6}}^{42575376056348197982869887629770563187} (6)\ d_{x}\\
&\times\int_{-\frac{1}{6}}^{33263421842796264531163131543015665831}(6)d_{y}\\
&=\displaystyle\left[6x+ C_{1} \right]_{\frac{1}{6}}^{42575376056348197982869887629770563187}\\  
& \ \ \times\displaystyle\left[6y+ C_{2} \right]_{- \frac{1}{6}}^{33263421842796264531163131543015665831}\\
&=\displaystyle\left(6\cdot 42575376056348197982869887629770563187-1 \right)\\ 
&\times \left(6\cdot 33263421842796264531163131543015665831+ 1 \right) \\
&=y(42575376056348197982869887629770563187)\\
& \ \ \ \times  y(33263421842796264531163131543015665831)\\
& \text{, from lemma \ref{lem1}}\\
&=255452256338089187897219325778623379121\\
&\ \ \ \times 199580531056777587186978789258093994987\\
&=509832969796079184022321203452975440752167\\
& \ \ \ 80739063868280880006751616374466427\\
&=n
\end{align*}
\end{itemize}
}
\end{example}

Additionally, in the same perspective, we have the following result, a new vulnerability on a RSA type modulus:

\begin{theorem} \ \\ \rm{
Le $n$ be a RSA modulus. If the diophantine equation $y^2-2^3n-1=0$ has any solution in $y$, then RSA is broken in constant time complexity $O(1)$ in best case or in logarithmic time complexity $O\left(\log(8n+1)\right)$ in the worst case.
}
\end{theorem}
\begin{proof} \ \\ \rm{
Let $n=pq$ a semiprime with $q<p$.\\
Assume $\displaystyle n=\sum_{l=1}^{p}l=1+2+3+ \ \cdots \ p$.\\ From the Gauss sum sequence, $n=S_{p}=\displaystyle p\frac{p+1}{2}$ or $n=\displaystyle\left( \begin{array}{c}
n+1 \\ 
2
\end{array}\right)$ in binomial coefficient representation. This represents the area of a rectangle whose width is $p$ and height is $\frac{p+1}{2}$.\\ In other words, $p$ verifies the equations $$y^{2}+y-2n=0 \ \ (1)$$ and $$x^2-\displaystyle\left(\frac{3p+1}{2}\right)x+\displaystyle n=0 \ \ (2)$$ from $x^2-Sx+P=0$ result.\\
Considering equation $(1)$, this equation yields integer roots if and only if its discriminant $1+2^3n$ is a square. Equivalently, there exists an integer $y$ such that $y^2-2^3n-1=0$. In this case $\displaystyle p=\frac{\lvert-1\pm\left(1+2^3n\right)^{1/2}\rvert}{2}$. Finally, depending on the square root computation algorithm, its complexity is at best in $O(1)$ or at worst in $O\left(\log(8n+1)\right)$.
}
\end{proof}
\textit{Remark} \ \\ \rm{
This theorem can be reformulated as follows:\\
If a RSA modulus is a triangular number, then RSA is broken in at least constant time complexity and at most logarithmic time complexity.
}

\begin{example} \ \rm{
\begin{itemize}
\item $n=15$, $8n+1=121$ is a square, then $p=\lvert 5 \vert=5$ is a prime factor of $n$: $n=3\times 5$.
\item $n=25651$, $8n+1=205209$ is a square, then $p=\lvert -227\rvert=227$ is a prime factor of $n$: $n=113\times 227$.
\item $n=16744225501$, $8n+1=133953804009$ is a square, then $p=\lvert -182999\rvert=182999$.
\item $n=22008842474653$, $8n+1=176070739797225$ is a square, then $p=\lvert -6634583\rvert=6634583$.
\item $n=21997001338003$, $8n+1=$ is a square, then $p=\lvert 6632797\rvert=6632797$.
\end{itemize}
}
\end{example}

\subsection{Matrix Decomposition point of view (MDpv)}
Here we present another approach, consisting in considering the modulus $n$ as a matrix determinant.
\begin{theorem} \ \\ \rm{
Let $n=pq$ be a RSA modulus, then there exists a matrix decomposition  $N$ of $n$ of size $2$ such that $N=PQ$ with $\det N=n$, $\det P=p$ and $\det Q=q$.
}
\end{theorem}
\begin{proof} \ \\ 
Triavially, one can find non zero integers $a$ and $b$ so that applying Bezout's theorem, there exists integers $d, c$ such that $ad-bc=\gcd\text{(}a, \ b\text{)}=n$.
Let $N=
\begin{pmatrix}
a & b \\ 
c & d
\end{pmatrix}$, $P=
\begin{pmatrix}
x_{1} & x_{2} \\ 
x_{3} & x_{4}
\end{pmatrix} $ and $Q=
\begin{pmatrix}
y_{1} & y_{2} \\ 
y_{3} & y_{4}
\end{pmatrix}$ $\in \displaystyle  M_{2}\left(\mathbb{Z}\right)$.
$n=pq$ if and only if \\ $\displaystyle ad-bc=\left(x_{1}x_{4}-x_{2}x_{3}\right) \left(y_{1}y_{4}-y_{2}y_{3}\right)$ or equivalently \\ $\displaystyle\det N=\det P \cdot \det Q$\\
$PQ=\begin{pmatrix}
x_{1} & x_{2} \\ 
x_{3} & x_{4}
\end{pmatrix} 
\begin{pmatrix}
y_{1} & y_{2} \\ 
y_{3} & y_{4}
\end{pmatrix}=
\begin{pmatrix}
x_{1}y_{1} \text{+} x_{2}y_{3} & x_{1}y_{2} \text{+} x_{2}y_{4} \\ 
x_{3}y_{1} \text{+} x_{4}y_{3} & x_{3}y_{2} \text{+} x_{4}y_{4}
\end{pmatrix}$
Then
\begin{eqnarray}\label{s1}
	x_{1}, x_{2}, x_{3}, x_{4}, y_{1}, y_{2}, y_{3}, y_{4} \text{ satisfy }
	\left\lbrace
	\begin{array}{ll}
	x_{1}y_{1} \text{+} x_{2}y_{3}=a\\
	x_{1}y_{2} \text{+} x_{2}y_{4}=b\\
	x_{3}y_{1} \text{+} x_{4}y_{3}=c\\
	x_{3}y_{2} \text{+} x_{4}y_{4}=d
	\end{array} \right.
\end{eqnarray}
From this system,\\
$\displaystyle x_{1}=\frac{ay_{4}-by_{3}}{y_{1}y_{4}-y_{2}y_{3}}$, $\displaystyle x_{2}=\frac{ay_{2}-by_{1}}{y_{1}y_{4}-y_{2}y_{3}}$, $\displaystyle x_{3}=\displaystyle\frac{cy_{4}-dy_{3}}{y_{1}y_{4}-y_{1}y_{4}}$, $x_{4}=\displaystyle\frac{cy_{2}-dy_{1}}{y_{1}y_{4}-y_{2}y_{3}}$
We verify that:
\begin{align*}
x_{1}x_{4}-x_{2}x_{3}&=\displaystyle\frac{ay_{4}-by_{3}}{y_{1}y_{4}-y_{2}y_{3}} \times \frac{dy_{1}-cy_{2}}{y_{1}y_{4}-y_{2}y_{3}}-\frac{cy_{4}-dy_{3}}{y_{1}y_{4}-y_{2}y_{3}}\times \frac{by_{1}-ay_{2}}{y_{1}y_{4}-y_{2}y_{3}}\\
&=\frac{\left(ay_{4}-by_{3}\right)\left(dy_{1}-cy_{2}\right)-\left(cy_{4}-dy_{3}\right)\left(by_{1}-ay_{2}\right)}{\left(y_{1}y_{4}-y_{2}y_{3}\right)^{2}}
\end{align*}
\begin{align*}
\ \ \ \ \ \ \ \ &=\displaystyle\frac{\left(bc-ad\right)y_{2}y_{3}\text{+}\left(ad-bc\right)y_{1}y_{4}\text{+}\left(ac-ac\right)y_{2}y_{4}\text{+}\left(bd-bd\right)y_{1}y_{3}}{\left(y_{1}y_{4}-y_{2}y_{3}\right)^{2}}\\
&=\displaystyle\frac{\left(bc-ad\right)y_{2}y_{3}\text{+}\left(ad-bc\right)y_{1}y_{4}}{\left(y_{1}y_{4}-y_{2}y_{3}\right)^{2}}\\
&=\displaystyle\frac{\left(ad-bc\right)\left(y_{1}y_{4}-y_{2}y_{3}\right)}{\left(y_{1}y_{4}-y_{2}y_{3}\right)^{2}}\\
&=\displaystyle\frac{ad-bc}{y_{1}y_{4}-y_{2}y_{3}}\\
&=\displaystyle\frac{n}{y_{1}y_{4}-y_{2}y_{3}}\\
&=x_{1}x_{4}-x_{2}x_{3}
\end{align*}
\end{proof}
From this, we have:
\begin{itemize}
\item $\displaystyle\gcd\left(\det P, \ n\right) \ \mid \ n$,\\
then $ \gcd\left(x_{1}x_{4}-x_{2}x_{3}, \ n \right) \ \mid \ n$, hence $\gcd\left(x_{1}x_{4}-x_{2}x_{3}, \ n \right)=y_{1}y_{4}-y_{2}y_{3}$\\
\item $\displaystyle\gcd\left(\det Q, \ n\right) \ \mid \ n$,\\
then $ \gcd\left(y_{1}y_{4}-y_{2}y_{3}, \ n \right) \ \mid \ n$, hence $\gcd\left(y_{1}y_{4}-y_{2}y_{3}, \ n \right)=x_{1}x_{4}-x_{2}x_{3}$
\end{itemize}
Hence we have the following new system from $\left(\ref{s1}\right)$.
\begin{eqnarray}
	\left\lbrace
	\begin{array}{ll}
	x_{1}y_{1} \text{+} x_{2}y_{3}=a\\
	x_{1}y_{2} \text{+} x_{2}y_{4}=b\\
	x_{3}y_{1} \text{+} x_{4}y_{3}=c\\
	x_{3}y_{2} \text{+} x_{4}y_{4}=d\\
	\left(x_{1}x_{4}-x_{2}x_{3}\right)\left(y_{1}y_{4}-y_{2}y_{3}\right) = n
	\end{array} \right.
\end{eqnarray}

Which is a non linear system of $5$ equations with $8$ unkowns. \\ \\
\textbf{System resolution} \ \\
In this subsubsection, we propose some approaches to solve this system.  \\ \\
\textbf{Gröbner Basis} \ \\
We apply gröbner basis to the resolution of this system of equations.\\
We start by setting the polynomials to zero as follows:
\begin{eqnarray}
	\left\lbrace
	\begin{array}{ll}
	x_{1}y_{1} \text{+} x_{2}y_{3}-a=0\\
	x_{1}y_{2} \text{+} x_{2}y_{4}-b=0\\
	x_{3}y_{1} \text{+} x_{4}y_{3}-c=0\\
	x_{3}y_{2} \text{+} x_{4}y_{4}-d=0\\
	\left(x_{1}x_{4}-x_{2}x_{3}\right)\left(y_{1}y_{4}-y_{2}y_{3}\right)-n=0
	\end{array} \right.
\end{eqnarray}
Since $x_{1}, x_{2}, x_{3}, x_{4}, y_{1}, y_{2}, y_{3}, y_{4}$ are integers, we consider the polynomial ring $\mathbb{Z}\left[x_{1}, x_{2}, x_{3}, x_{4}, y_{1}, y_{2}, y_{3}, y_{4} \right]$.\\
Another consideration is the lexicographic order as monomial order given by $$x_{1}>x_{2}>x_{3}>x_{4}>y_{1}>y_{2}>y_{3}>y_{4}$$ and finally we consider the ideal $I$ generated by $f_{1}, f_{2}, f_{3}, f_{4}$ and $f_{5}$ as follows $I=\left\langle f_{1}, f_{2}, f_{3}, f_{4}, f_{5} \right\rangle$ over $\mathbb{Z}\left[x_{1}, x_{2}, x_{3}, x_{4}, y_{1}, y_{2}, y_{3}, y_{4} \right]$ where 
 $f_{1}=x_{1}y_{1} \text{+} x_{2}y_{3}-a$;\\ $f_{2}=x_{1}y_{2} \text{+} x_{2}y_{4}-b$; \ $f_{3}=x_{3}y_{1} \text{+} x_{4}y_{3}-c$;\\ $f_{4}=x_{3}y_{2} \text{+} x_{4}y_{4}-d$; and $f_{5}=\left(x_{1}x_{4}-x_{2}x_{3}\right)\left(y_{1}y_{4}-y_{2}y_{3}\right)-n$.\\
We then apply the Buchberger's algorithm to generate the gröbner basis.\\
Let $G$ be the gröbner basis, we initialize $G$ as $G=\displaystyle\lbrace f_{1}, f_{2}, f_{3}, f_{4}, f_{5}\rbrace$ and we compute the S-polynomials as follows: for each pair of polynomials $f_{i}$ and $f_{j}$ in $G$, 
$$S\displaystyle\left(f_{i}, f_{j}\right)=\frac{Lcm\left(Lm(f_{i}), Lm(f_{j}) \right)}{Lt(f_{i})}f_{i}-\frac{Lcm\left(Lm(f_{i}), Lm(f_{j}) \right)}{Lt(f_{j})}f_{j}$$
 where $Lcm$ is the least common multiple, $Lt$ the leading term and $Lm$ the leading monom.\\
Since our ideal $I$ is generated by $5$ polynomials, we expect the set $S$ of S-polynomials to have a cardinal of $Card\left(S\right)\displaystyle= \begin{pmatrix}
5 \ \\ 
2 \
\end{pmatrix}=\frac{5!}{2!3!}=10$.
Computation of S-Polynomials:\\
$Lt(f_{1})=Lm(f_{1})=x_{1}y_{1}$; \ \ \ \ \ \ \ $Lt(f_{2})=Lm(f_{2})=x_{1}y_{2}$;\\ $Lt(f_{3})=Lm(f_{3})=x_{3}y_{1}$; \ \ \ \ \ \ \ 
$Lt(f_{4})=Lm(f_{4})=x_{3}y_{2}$;\\
$Lt(f_{5})=Lm(f_{5})=x_{1}x_{4}y_{1}y_{4}$.\\
We get:\\
$S\displaystyle\left(f_{1}, f_{2}\right)=y_{2}f_{1}-y_{1}f_{2}$; \ \ \ \ \ \ \ \ \ \
$S\displaystyle\left(f_{1}, f_{3}\right)=x_{3}f_{1}-x_{1}f_{3}$; \\

$S\displaystyle\left(f_{1}, f_{4}\right)=x_{3}y_{2}f_{1}-x_{1}y_{1}f_{4}$; \ \ \ \ \ \ \ \ \ \
$S\displaystyle\left(f_{1}, f_{5}\right)=x_{4}y_{4}f_{1}-f_{5}$; \\

$S\displaystyle\left(f_{2}, f_{3}\right)=x_{3}y_{1}f_{2}-x_{1}y_{2}f_{3}$;  \ \ \ \ \ \ \ \ \ \
$S\displaystyle\left(f_{2}, f_{4}\right)=x_{3}y_{1}f_{2}-x_{1}y_{1}f_{4}$; \\

$S\displaystyle\left(f_{2}, f_{5}\right)=x_{4}y_{1}y_{4}f_{2}-y_{2}f_{5}$; \ \ \ \ \ \ \ \ \ \
$S\displaystyle\left(f_{3}, f_{4}\right)=y_{2}f_{3}-y_{1}f_{4}$; \\

$S\displaystyle\left(f_{3}, f_{5}\right)=x_{1}x_{4}y_{4}f_{3}-x_{3}f_{5}$; \ \ \ \ \ \ \ \ \ \
$S\displaystyle\left(f_{4}, f_{5}\right)=x_{1}x_{4}y_{1}y_{4}f_{4}-x_{3}y_{2}f_{5}$;\\
In the polynomial ring $\mathbb{Z}\left[x_{1}, x_{2}, x_{3}, x_{4}, y_{1}, y_{2}, y_{3}, y_{4} \right]$, a Gröbner basis is:
\vspace{-0.3cm}
\begin{align*}
\displaystyle &ax_{3} - cx_{1} + x_{1}x_{4}y_{3} - x_{2}x_{3}y_{3}=0 \ (1),\\ 
&bx_{3} - dx_{1} + x_{1}x_{4}y_{4} - x_{2}x_{3}y_{4}=0 \ (2),\\ 
&-a + x_{1}y_{1} + x_{2}y_{3}=0, -b + x_{1}y_{2} + x_{2}y_{4}=0 \ (3),\\ 
&a^{2}x_{3}y_{4} - abx_{3}y_{3} - acx_{1}y_{4} + bcx_{1}y_{3} + nx_{1}y_{3}=0 \ (4),\\ 
&ax_{3}y_{4} - bx_{3}y_{3} - cx_{1}y_{4} + dx_{1}y_{3}=0 \ (5),\\ &ay_{2} - by_{1} + x_{2}y_{1}y_{4} - x_{2}y_{2}y_{3}=0 \ (6),\\
 &a^{2}x_{4}y_{2} - abx_{4}y_{1} - acx_{2}y_{2} + bcx_{2}y_{1} + nx_{2}y_{1}=0 \ (7),\\ 
&ax_{4}y_{2} - bx_{4}y_{1} - cx_{2}y_{2} + dx_{2}y_{1}=0 \ (8),\\ &a^{2}x_{4}y_{4} - abx_{4}y_{3} - acx_{2}y_{4} - an + bcx_{2}y_{3} + nx_{2}y_{3}=0 \ (9),\\ 
&ax_{4}y_{4} - bx_{4}y_{3} - cx_{2}y_{4} + dx_{2}y_{3} - n=0 \  (10),\\ 
&-c + x_{3}y_{1} + x_{4}y_{3}=0, -d + x_{3}y_{2} + x_{4}y_{4}=0 \ (11),\\ 
&cy_{2} - dy_{1} + x_{4}y_{1}y_{4} - x_{4}y_{2}y_{3}=0 \ (12),\\
 &ad - bc - n=0 \ (13)
\end{align*}
given in term of a system of equations as
\begin{eqnarray}\label{syst}
	\left\lbrace
	\begin{array}{ll}
	\left\lbrace
	\begin{array}{ll}
	ad - bc - n=0\\
	x_{3}y_{1} + x_{4}y_{3}-c=0\\
	x_{1}y_{1} + x_{2}y_{3}-a=0 \ \ \ \ \ \ \ \ \ \ \ \ \ (3)\\
	x_{1}y_{2} + x_{2}y_{4}-b=0\\
	x_{3}y_{2} + x_{4}y_{4}-d=0\\
	\end{array} \right.\\
	cy_{2} - dy_{1} + x_{4}y_{1}y_{4} - x_{4}y_{2}y_{3}=0\\
	ax_{3} - cx_{1} + x_{1}x_{4}y_{3} - x_{2}x_{3}y_{3}=0\\
	ay_{2} - by_{1} + x_{2}y_{1}y_{4} - x_{2}y_{2}y_{3}=0\\
	bx_{3} - dx_{1} + x_{1}x_{4}y_{4} - x_{2}x_{3}y_{4}=0\\
	ax_{3}y_{4} - bx_{3}y_{3} - cx_{1}y_{4} + dx_{1}y_{3}=0\\
	ax_{4}y_{2} - bx_{4}y_{1} - cx_{2}y_{2} + dx_{2}y_{1}=0\\
	ax_{4}y_{4} - bx_{4}y_{3} - cx_{2}y_{4} + dx_{2}y_{3} - n=0\\
	a^{2}x_{3}y_{4} - abx_{3}y_{3} - acx_{1}y_{4} + bcx_{1}y_{3} + nx_{1}y_{3}=0\\
	a^{2}x_{4}y_{2} - abx_{4}y_{1} - acx_{2}y_{2} + bcx_{2}y_{1} + nx_{2}y_{1}=0\\
	a^{2}x_{4}y_{4} - abx_{4}y_{3} - acx_{2}y_{4} - an + bcx_{2}y_{3} + nx_{2}y_{3}=0\\
	\end{array} \right.
\end{eqnarray}
\textbf{Discussion}:\\
A part from trying to find factors $p$ and $q$ through system resolution, one can for different matrix decompositions $N$ of $n$, diagonalize $N$ if $(a+d)^2-4n$ is a square. Indeed the characteristic polynomial in this case $P(\lambda)=(a-\lambda)(d-\lambda)-bc=\lambda^2-(a+d)\lambda+n$, which is diagonalizable over $\mathbb{Z}$ only if the discriminant of the characteristic polynomial is a square. 

\begin{example} \ \\ \rm{
Let $n=9829760426531292080598278473834159188865611141619845\\3609045170174910824769389$
with $a=98297604265312920805982784\\738341591889580784184957701040844308085531017447603$; \\
$b=9829760426531292080598278473834159189017411569821850616\\1959832718260667360288$\\
$c=331341255498442310683613277890542765532$\\
$d=331341255498442310683613277890542765535$
$N=
\begin{pmatrix}
a & b \\ 
c & d
\end{pmatrix}$
}. Considering the System \ref{syst}, a solution of the system gives:
$x_{1}=296665756630402560557762316364824956342$; \\
$x_{2}=296665756630402560557762316364824956343$;\\
$x_{3}=1$; $x_{4}=2$.\\
$y_{1}=331341255498442310683613277890542765530$; \\
$y_{2}=331341255498442310683613277890542765531$; \\
$y_{3}=1$; $y_{4}=2$; Hence 
$P=
\begin{pmatrix}
x_{1} & x_{2} \\ 
x_{3} & x_{4}
\end{pmatrix} $ and $Q=
\begin{pmatrix}
y_{1} & y_{2} \\ 
y_{3} & y_{4}
\end{pmatrix}$ $\in \displaystyle  M_{2}\left(\mathbb{Z}\right)$,\\
 $p=\det(P)=x_{1}x_{4}-x_{2}x_{3}=29666575663040256055776231636482495\\6341$ and $q=\det(Q)=y_{1}y_{4}-y_{2}y_{3}=331341255498442310683613277\\890542765529$
\end{example}
\subsection{Modulus Algebraic Form Point of View (MAFpv)}
In this subsection, we present another approach that consists in considering different algebraic forms of the semiprime factors.
\begin{theorem} \ \\ \rm{
Any RSA modulus, particularly any semiprime $n$ can be represented in the ring $\mathbb{Z}[x, y]$ as one of the following:
\begin{eqnarray*}
n(x, y)=
\left\lbrace
\begin{array}{ll}
36xy-6(x+y)+1\\
36xy+6(x+y)+1\\
36xy+6(x-y)-1\\
36xy-6(x-y)-1
\end{array} \right.
\end{eqnarray*}
}
\end{theorem}

\begin{proof} \ \\ \rm{
From Lemma \ref{lem1}, there exist differential equations $t^{'}=6$ with $t(0)=\pm 1$ and $z^{'}=6$ with $z(0)=\pm 1$ such that $n(x, y)=z(x)t(y)$.
We then distinguish the following cases:
\begin{description}
\item[If $t(0)=-1$ and $z(0)=-1$]: $t(y)=6y-1$ and $z(x)=6x-1$. Then $n(x,y)=z(x)t(y)=(6x-1)(6y-1)=36xy-6(x+y)+1$.
\item[If $t(0)=1$ and $z(0)=1$]: $t(y)=6y+1$ and $z(x)=6x+1$. Then $n(x,y)=z(x)t(y)=(6x+1)(6y+1)=36xy+6(x+y)+1$.
\item[If $t(0)=-1$ and $z(0)=1$]: $t(y)=6y-1$ and $z(x)=6x+1$. Then $n(x,y)=z(x)t(y)=(6x+1)(6y-1)=36xy-6(x-y)-1$.
\item[If $t(0)=1$ and $z(0)=-1$]: $t(y)=6y+1$ and $z(x)=6x-1$. Then $n(x,y)=z(x)t(y)=(6x-1)(6y+1)=36xy+6(x-y)-1$.
\end{description}

All these can be re-written as:
\begin{eqnarray*}
n(x, y)=
\left\lbrace
\begin{array}{ll}
36xy-6(x+y)+1\\
36xy+6(x+y)+1\\
36xy+6(x-y)-1\\
36xy-6(x-y)-1
\end{array} \right. \text{or } 
n(x, y)=
\left\lbrace
\begin{array}{ll}
36xy+6\left|x+y\right|+1\\
36xy+6\left|x-y \right|-1
\end{array} \right.
\end{eqnarray*}
}
\end{proof}

\begin{remark} \ \rm{
\begin{description}
\item[$(1)$] The polynomials $n(x, y)\in\mathbb{Z}[x, y]$ as defined in the previous theorem have a total degree of $d=2$, ie every monomial $x^{i}y^{j}$ satisfies $i+j \leq d=2$
\item[$(2)$] We consider $X\in\mathbb{Z}_{>0}$ respectively $Y\in\mathbb{Z}_{>0}$, the upper bounds for $x$ respectively for $y$ and $M$ a RSA type modulus greater than $n$.
\end{description}
}
\end{remark}

From this remark, we have the following theorems:
\begin{theorem} \ \\ \rm{
Let $\displaystyle\deg_{x}\left(n(x, y)\right), \deg_{y}\left(n(x, y)\right) < d=2$;
\begin{eqnarray*}
n(x, y)=\displaystyle\sum_{0\leq i, j < d}n_{i,j}\delta_{ij}x^{i}y^{j}=
\left\lbrace
\begin{array}{ll}
36xy+6\left|x+y\right|+1\\
36xy+6\left|x-y \right|-1
\end{array} \right.
\end{eqnarray*}
where $n_{00}=1$, $n_{11}=36$, $n_{10}=n_{01}=6$ and \begin{eqnarray*}
\delta_{ij}=
\left\lbrace
\begin{array}{ll}
1 \ \text{ if } i=j\\
0 \ \text{ else }
\end{array} \right.
\end{eqnarray*}


and for $\vert x \vert <X$ and $\vert y \vert <Y$, define $W=\displaystyle\max_{0\leq i, j \leq d}\left|n_{i, j}\right|X^{i}Y^{j}$.\\
If $XY<W^{1/3}=6^{2/3}X^{1/3}Y^{1/3}$ then we can recover solutions $(x, y)$ in polynomial time in $\log(W)$ and $2^{2}$.
}
\end{theorem}
\begin{proof} \ \\ \rm{
This is a direct application of coppersmith's theorem on bivariate polynomials \cite{coppersmith}.
}
\end{proof}
From this first underlying theorem, the condition to recover solutions $(x, y)$ in polynomial time cannot be achieved in practice. Hence, we consider the same theorem in a residual ring $\mathbb{Z}_{M}$ with $M$ a RSA type modulus greater than $n$.

\begin{theorem} \ \\ \rm{
Consider the polynomial $n(x, y)\in\mathbb{Z}[x, y]$ as defined above. Let $M$ be a RSA type modulus greater than $n$, $X$ and $Y$ such that for all $(x, y)$ solutions of $n(x, y)$,   $\vert x \vert <X$ and $\vert y \vert <Y$. If $XY<M^{1/2-\epsilon}$ for some $0<\epsilon<1/2$ then we can compute $n(x, y) \equiv 0\mod M$ and find solutions of $n(x, y)$ over $\mathbb{Z}$.
}
\end{theorem}
\begin{proof}
As with the previous theorem, this is a direct application of coppersmith's theorem on bivariate polynomials \cite{coppersmith} on a residual ring.
\end{proof}

%% file: conclusion.tex
\section{Conclusion}